\documentclass[reqno]{amsart}

\usepackage{amsfonts}
\usepackage{amssymb}
\usepackage{amsmath}
\usepackage{mathrsfs}
\usepackage{amsthm}
\usepackage[colorlinks,linkcolor=blue,citecolor=blue]{hyperref}
\usepackage{enumitem}

\newtheorem{theorem}{Theorem}[section]
\newtheorem{corollary}[theorem]{Corollary}
\newtheorem{lemma}[theorem]{Lemma}
\newtheorem{proposition}[theorem]{Proposition}

\theoremstyle{remark}
\newtheorem{remark}[theorem]{Remark}

\newcommand{\abs}[1]{\left\lvert{#1}\right\rvert}
\newcommand{\norm}[1]{\left\lVert{#1}\right\rVert}
\newcommand{\dd}{\ensuremath{\text{d}}}
\newcommand{\R}{\ensuremath{\mathbb{R}}}
\newcommand{\N}{\ensuremath{\mathbb{N}}}
\renewcommand{\H}{{\mathbb{H}}}

\newcommand{\pd}[2]{\frac{\partial {#1}}{\partial {#2}}}
\newcommand{\ip}[2]{\ensuremath{\left< {#1}, {#2} \right>}}
\newcommand{\eps}{\ensuremath{\varepsilon}}

\newcommand{\bnabla}{\overline{\nabla}}
\newcommand{\bg}{\bar{g}}

\newcommand{\diag}[1]{\ensuremath{\operatorname{diag}\{ #1 \}}}

\theoremstyle{definition}
\newtheorem{definition}[theorem]{Definition}

\numberwithin{equation}{section}

\begin{document}

\title[Inverse curvature flow]
{Inverse Hessian Curvature Flow in Minkowski Space I: Cocompact Hypersurfaces}

\author[D. Li]{Dake Li}
\address[D. Li]{School of Mathematical Science\\
Fudan University\\
200433 Shanghai\\
P.R. China}
\email{dkli25@m.fudan.edu.cn}

\author[Z. Wang]{Zhizhang Wang}
\address[Z. Wang]{School of Mathematical Science\\
Fudan University\\
200433 Shanghai\\
P.R. China}
\email{zzwang@fudan.edu.cn}

\author[S. Yin]{Shiqi Yin}
\address[S. Yin]{School of Mathematical Science\\
Fudan University\\
200433 Shanghai\\
P.R. China}
\email{sqyin25@m.fudan.edu.cn}


\thanks{The second author is supported by NSFC Grant No.12141105.}


\begin{abstract}
	In this paper, we investigate the cocompact inverse \(\sigma_k\) curvature flow in Minkowski space. We prove the longtime existence and convergence of this flow. As a consequence, a quermassintegral inequality is established. 
\end{abstract}

\maketitle 

\section{Introduction}
Minkowski space \(\R^{n,1}\) is  \(\R^{n+1}\) endowed with the pseudo-metric 
\[ ds^2 = dx_1^2 +\cdots + dx_n^2 - dx_{n+1}^2. \] i.e. \(\ip{X}{Y} = \sum_{i=1}^{n}x_i y_i - x_{n+1}y_{n+1}\). A nonzero vector \(X \in \R^{n,1}\) is said to be spacelike, timelike, or null respectively if \(\ip{X}{X} > 0, < 0\) or \(=0\) respectively. Denote \(E_{n+1} = (0,\cdots,0,1)\in \R^{n+1}\). We regard \(x_{n+1}\) as the time function. If the  \((n+1)\)-th component of \(X\) is positive, or equivalently, \(\ip{X}{E_{n+1}} < 0\), we say that \(X\) is future directed. 
\par In this paper, we study the inverse Hessian curvature flow. 

\begin{definition}[ICF in Minkowski space]
    A family of time-dependent spacelike hypersurfaces \(M_t = X:M\times [0,T) \to \R^{n,1}\) is said to be an inverse curvature flow if it satisfies following equations:
    \begin{equation}\label{eq:ICF}
        \left\{
        \begin{aligned}
        &\pd{}{t}X(\cdot,t) = -\frac{1}{F(\cdot,t)}\nu(\cdot,t),\quad \forall t\in (0,T), \\
        &X(\cdot,0)=X_0(\cdot), 
        \end{aligned}
        \right.
    \end{equation}
    where \(F = F(h,g) = f(\kappa_1,\cdots,\kappa_n)\) is  the curvature function depending on the first and second fundamental forms $g,h$, or the principal curvatures $\kappa=(\kappa_1,\cdots,\kappa_n)$ of $M_t$.  
 Moreover, $f(\kappa)$ is strictly increasing with respect to each \(\kappa_i\), concave, and of homogeneous degree \(1\), normalized by \(f(1,1,\cdots,1) = 1\). \(\nu\) is the future directed (timelike) unit normal vector, and \(F > 0\) for all \((p,t)\in M\times [0,T)\). 
\end{definition}
In this paper we only study \(F = s_k^{1/k}\equiv (\sigma_k / \binom{n}{k})^{1/k}\), \(k=1,2.\cdots,n\), where the Hessian curvature $\sigma_{k}$ is defined by the $k$-th  elementary symmetric polynomial of the principal curvatures:
\begin{equation*}
	\sigma_{k}(\kappa)=\sum_{1\leq i_{1}< i_{2}<\cdots< i_{k}\leq n}\kappa_{i_{1}}\kappa_{i_{2}}\cdots\kappa_{i_{k}}.
\end{equation*}
For example, if \(F(p,t) = s_1(\kappa)(p,t) = \frac{g^{ij}h_{ij}}{n}\) is the mean curvature, then \eqref{eq:ICF} is called the \textit{inverse mean curvature flow}.

The curvature flow problem in Euclidean space has been extensively studied in the literature. In \cite{Hui}, Huisken studied the mean curvature flow, which has been generalized to Gauss curvature and a large family of curvature flows by Chow \cite{Chow} and Andrews \cite{Pinching-estimates}.  In \cite{Ger} and \cite{Ur}, Gerhardt and Urbas respectively considers the inverse curvature flow in Euclidean space, which is a similar flow equation as \eqref{eq:ICF} with the opposite sign. Roughly speaking, they have proved that, for any compact initial hypersurface $M_0$, the ICF will expand $M_0$ and it will converge to some sphere by a suitable scaling. For entire graphical flows,  Daskalopoulos-Huisken \cite{DH} studied the inverse mean curvature flow. 

In Minkowski space, Andrews, Chen, Fang and McCoy \cite{co-compact} first studied the curvature flow for cocompact hypersurfaces. More precisely, they consider the curvature flow:   
$$\pd{}{t}X(\cdot,t) = F(\cdot,t)\nu(\cdot,t),\quad \forall t\in (0,T), $$
 where $F$ is a homogenous degree 1 curvature function defined in Definition  \ref{eq:ICF}     
 and the initial hypersurface $M_0$ is spacelike, co-compact, and strictly convex. Aarons \cite{Aa}, Bayard-Schn\"{u}rer \cite{BS09} and Bayard \cite{B09} studied the mean curvature flow, Gauss curvature flow and scalar curvature flow with a forcing term in Minkowski space. For the non-compact hypersurfaces, the second author and Xiao \cite{WX22-1,WX22-2} studied the Hessian curvature self-expanders and the Hessian curvature flows of entire space-like hypersurfaces in Minkowski space. Then, Qu, the second author and Wo \cite{QWW24}  generalized a similar result to Hessian curvature translating solutions. Moreover, the second author and Xiao \cite{WX212} also studied some fully nonlinear curvature flows of non-compact spacelike hypersurfaces in Minkowski space. One may see \cite{RenWX19, RenWX20, WX20,WX21} for related prescribed curvature problem in Minkowski space.  

An important application of the inverse curvature flow is to prove quermassintegral inequalities. By using  \cite{Ger, Ur}, Guan-Li \cite{Guan2009} establishes the quermassintegral inequalities for star-shaped compact hypersurfaces in Euclidean space.  
In hyperbolic space, the inverse curvature flow and 
the quermassintegral inequalities were first studied by Wang-Xia \cite{WXia}, also see Hu-Li-Wei \cite{HLWei}. 
Brendle-Guan-Li \cite{BGL}  designed a locally constrained inverse curvature flow and studied the quermassintegral inequalities in space forms. In de Sitter space, Hu-Li \cite{HuLi} establishes the full quermassintegral inequalities for compact strictly hypersurfaces if its principal curvatures are less than 1. If we would like to remove the upper bound of the principal curvatures, they can prove part of quermassintegral inequalities. For more reference, the reader
may see \cite{CPLS, Sch, MaK, MaK2} and the reference therein.

\par In this paper, we discuss the ICF of standard co-compact hypersurfaces. Following \cite{co-compact}, we have 
\begin{definition}[standard co-compact]\label{defn:std co-compact}
A spacelike hypersurface $M$ in $\mathbb{R}^{n,1}$ is said to be standard co-compact if there exists a discrete subgroup $G$ of the future preserved Lorentz group $\mathrm{O}_+(\mathbb{R}^{n,1})$ that acts freely and properly discontinuously on $M$, such that the quotient Riemannian manifold $M/G$ is compact. Here 
$$ \begin{aligned}
    &\quad \mathrm{O}_+(\mathbb{R}^{n,1}) = \\
    & \Big\{\text{linear } T:\R^{n,1}\to \R^{n,1} \mid \ip{TX}{TY} = \ip{X}{Y}, \forall X,Y\in \R^{n,1}, \ip{TE_{n+1}}{E_{n+1}} < 0 \Big\}
\end{aligned} $$
is the future preserved Lorentz group. 
\end{definition} 
 All orientable closed surfaces with negative Gauss curvature can be isometric embedded into Minkowski space $\mathbb{R}^{2,1}$ as some spacelike strictly convex standard co-compact surfaces by a theorem of Chen-Yin \cite{ChYin}. Therefore, spacelike strictly convex standard cocompact hypersurfaces can be viewed as closed manifolds with negative sectional curvatures. Moreover, these manifolds have nontrivial topological structure induced by the group $G$.

\par A special class of solutions to equation \eqref{eq:ICF} are the self-shrinkers, which are defined by 
\begin{definition}[self-shrinker]
    A spacelike hypersurface \(X:M\to \R^{n,1}\) is called a self-shrinker of \eqref{eq:ICF} if it satisfies the equation 
    \begin{equation}\label{eq:self-shrinker}
        F = -\frac{1}{\ip{X}{\nu}}.
    \end{equation}
\end{definition}

It is obvious that the hyperboloid \(\H^n = \{X\in \R^{n,1}\mid \ip{X}{X} = -1\}\) satisfies the above equation. The self-shrinker defined here is corresponding to the self-shrinker of curvature flow for closed compact hypersurfaces in Euclidean space.  
We can prove the following rigidity result for any co-compact self-shrinker: 
\begin{theorem}\label{lem:rigidity}
    A $k$-convex co-compact spacelike self-shrinker of the ICF \eqref{eq:ICF}, where \(F = s_k^{1/k}\), must be a branch of hyperboloid \(\H^n\). 
\end{theorem}
 Here $k$-convex means: 
 \begin{definition}
\label{intdef1}
A $C^2$ regular hypersurface $M\subset \R^{n, 1}$ is $k$-convex,
if the principal curvatures of $M$ at $X\in M$ satisfy $\kappa[X]\in\Gamma_k$ for all $X\in M$, where $\Gamma_k$
is the G\r{a}rding cone
\[\Gamma_k=\{\lambda\in\R^n|\sigma_m(\lambda)>0, m=1, \cdots, k\}.\]
\end{definition} 
 
 One may compare Theorem  \ref{lem:rigidity} with Gao-Li-Wang \cite{GaoLiW}. 
Let 
$$u(x,t)=\sqrt{e^{-2t}+|x|^2}.$$  
One may easily check that $X_t=(x,u(x,t))$ satisfying \eqref{eq:ICF}. As $t$ goes to infinity, we can see the hypersurface $M_t$ convergences to the light cone, whose principal curvatures are infinite. Now we take a normalized flow 
$$\tilde{X}_t=e^t X_t,$$   
then $\tilde{X}_t$ convergences to the hyperboloid. In general, similar to the above baby model, we can prove:   

\begin{theorem}\label{thm:convergence}
    Let \(F = s_k^{1/k}\) in the ICF \eqref{eq:ICF}, where \(s_k = {\sigma_k}/{\binom{n}{k}}\) is the normalized k-th elementary symmetric polynomial. If the initial spacelike hypersurface is strictly convex and standard co-compact, then the flow \eqref{eq:ICF} has a unique long-time solution \(X: M\times [0,+\infty)\to \R^{n,1}\). After scaling by \(\tilde{X}(p,t) = e^t X(p,t)\), this solution converges in \(C^{\infty}\)-topology to the upper branch of a hyperboloid \(\H^n(r_\infty)\) with radius \(r_{\infty}\). In addition, the limit hyperboloid \(M^{\infty}\) satisfies the volume relation: 
    \begin{equation*}
        \operatorname{Vol} (M^{\infty}/G) \leq \operatorname{Vol} (M_0/G),
    \end{equation*}
    with equality when \(F(\kappa) = s_1(\kappa)\). 
\end{theorem}

Similar to Guan-Li \cite{Guan2009}, we can prove the quermassintegral inequalities by using the above theorem:
\begin{theorem}\label{thm:quermassintegral_ineq}
    Let \(M\) be a standard co-compact convex spacelike hypersurface with related Lorentz group \(G\), \(\overline{M} = M / G\). Let \(\abs{\overline{M}} := \int_{\overline{M}} 1 \dd \mu_M\) be the \(n\)-dimensional volume of $\overline{M}$. Then for \(1\leq k \leq n-1\), we have
    \begin{equation}\label{eq:geo-ineq}
         {\left( \int_{\overline{M}} \sigma_k(\kappa) \dd \mu_M \right)^{\frac{1}{n-k}}} \leq {(C_n^k)^{\frac{1}{n-k}}\abs{\H^n / G}^{\frac{1}{n-k} - \frac{1}{n}}} \cdot \abs{\overline{M}}^{\frac{1}{n}}.
    \end{equation}
    The equality holds if and only if \(M\) is exactly a part of \(\H^n\). 
\end{theorem}

\begin{remark}
When $k=n$, the inequality \eqref{eq:geo-ineq}  becomes
\begin{equation}
\int_{\overline{M}} \sigma_n(\kappa) d\mu_M = |\mathbb{H}^n/G|.
\end{equation}
Geometrically, $\sigma_n(\kappa)d\mu_M$ corresponds to the volume form of the hyperbolic space $\mathbb{H}^n$ pulled back via the Gauss map $N: M \to \mathbb{H}^n$. Since the Gauss map induces a global diffeomorphism between the compact quotients $\overline{M}$ and $\mathbb{H}^n/G$, the total integral of the Gauss-Kronecker curvature $\sigma_n(\kappa)$ equals to the volume of $\mathbb{H}^n/G$, which is called the {\em hyperbolic volume}.
\end{remark}

\begin{remark}
    By Mostow's theorem (see \cite{Benedetti1992}), the hyperbolic volume \(\abs{\H^n / G}\) on the right-hand side of \eqref{eq:geo-ineq} is a topological invariant of \( M/G \).
\end{remark}
\begin{remark}
    When \(n\) is an even number, the hyperbolic volume can be explicitly expressed by the Euler Characteristic (\cite{Ratcliffe2019} Theorem 11.3.2, Gauss-Bonnet Theorem)
    \[ |\H^n / G| = \frac{(-1)^{\frac{n}{2}} \cdot \text{Vol}(S^{n})}{2} \chi(\H^n / G), \]
    where \(\text{Vol}(S^{n})\) is the Euclidean volume of the $n$-dimensional unit sphere. Since \(\H^n / G\) is homeomorphic to \(M / G\) under the flow, \(\chi(\H^n / G) = \chi(M / G)\).
When \(n\) is an odd number, we do not have an explicit formula for the hyperbolic volume. However, it is still a topological invariant of \(M / G\). 
\end{remark}

\begin{remark}\label{GuanYuan}
    As Guan and Yuan \cite{guan2025-intrinsicality} pointed, if \(M^n\) is a hypersurface in  \(\mathbb{R}^{n+1}\), each \(\sigma_k(\kappa[M])\) is intrinsic. We will give an analogous result in Minkowski space (Lemma \ref{lem:intrincity}). 
\end{remark}

Since all known quermassintegral inequalities hold for topological spheres, we think our inequalities are the first quermassintegral inequalities for manifolds with nontrivial topological structure, i.e. ``high genus" manifolds.  Therefore, we believe that inequalities \eqref{eq:geo-ineq}  are  interesting .          

At last, let's give two questions. 
A compact manifold with negative sectional curvature is not necessarily able to be embedded in Minkowski space as a spacelike hypersurface of codimension \(1\). Moreover, by Remark \ref{GuanYuan}, \(\sigma_k(\kappa)\) can be expressed as a polynomial of the Riemannian curvature tensor \(Riem_g\), then the integral in \eqref{eq:geo-ineq} is still meaningful.
Therefore, it is a natural question that 
does the geometric inequalities \eqref{eq:geo-ineq} hold for all closed manifolds \((N,g)\) with negative sectional curvature? 

Another question is that can we extend Theorem \ref{thm:convergence} to $F=\frac{\sigma_k}{\sigma_{k-1}}$ and then give the full quermassintegral inequalities as \cite{Guan2009} ? The difficulty here is that we can not give the upper bound of the principal curvatures in Minkowski space. In our second paper \cite{LiWYin}, we can give some counterexamples for non cocompact hypersurfaces. In fact, we can construct some strictly convex initial hypersurface $M_0$, and the flow equation has a solution $M_t$ such that, at finite time, the curvatures will blow up. However, we do not find appropriate counterexample for cocompact case. Note that, in \cite{HuLi},  a similar difficulty occurs in de Sitter space.


The organization of the paper is as follows. Section \ref{sec:rigidity} gives the proof of rigidity of self-shrinkers. In Section \ref{sec:evolution}, we will derive the evolution equations of geometric quantities. In Section \ref{sec:gauss-map}, we will rewrite our flow equation in hyperboloid by Gauss map. In Section \ref{sec:unnormalized-flow}, we establish basic estimates for the unnormalized flow, including the bounds for the support function, the speed function, and the gradient of the support function. In Section \ref{sec:normalized-flow}, we derive the normalized flow equation and establish the uniform curvature estimates, which subsequently guarantees the long-time existence of the normalized and unnormalized flows. Section \ref{sec:convergence} proves the convergence of the normalized flow. In the last Section, we establish the quermassintegral inequalities and prove that $\sigma_k$ curvatures also are intrinsic in Minkowski space.

\section{Rigidity of Self-shrinkers}\label{sec:rigidity}

Let \(\{x_1,\cdots,x_n\}\) be a local coordinate of \(M\). Denote \(X_i := \pd{X}{x_i}, X_{ij} := \frac{\partial^2 X}{\partial x_i \partial x_j}, \nu_j \) \(:= \pd{\nu}{x_j}\), then we have the Gauss-Weingarten formula: 
\begin{equation}\label{eq:GW-formula}
    \left\{
    \begin{aligned}
    &X_{ij} = \Gamma_{ij}^k X_k + h_{ij}\nu, \\
    &\nu_j = A_j^k X_k, 
    \end{aligned}
    \right.
\end{equation}
where \(A_j^k = g^{kl}h_{lj}\) is the shape operator,   $\Gamma_{ij}^k$ is the Christoffel symbol and \(\nu\) is the future directed (timelike) unit normal vector.

The main part of this section is to prove Theorem \ref{lem:rigidity}.  

\textbf{Proof of Theorem }{\ref{lem:rigidity}}: 
    Let \(\phi := \ip{X}{\nu}\). We will show that 
    \[ \sum_{i,j=1}^{n} F^{ij}\phi_{ij} \leq \frac{1}{\phi^2}\ip{X}{\nabla \phi}, \]
    where \(F^{ij} = \pd{F}{h_{ij}}\). For a fixed point \(p\in M\), choose a local orthonormal frame \(\{e_1,\cdots,e_n\}\) such that \(\nabla e_j(p) = 0, \forall j=1,2,\cdots,n\) and \(F^{ij}(p) = \diag{F^{11},\cdots,F^{nn}}\). Then, at \(p\) we have 
    \[ \begin{aligned}
        \phi_i :&= \nabla_{e_i}\phi = \sum_{j=1}^n h_{ij}\ip{X}{e_j},\\
        \phi_{ii} &= \sum_{j=1}^{n}\nabla_i h_{ij}\ip{X}{e_j} + \sum_{j=1}^{n}h_{ij}\ip{e_i}{e_j} + \sum_{j=1}^{n}h_{ij}h_{ji}\ip{X}{\nu} \\
        &= \sum_{j=1}^{n}\nabla_j h_{ii}\ip{X}{e_j} + h_{ii} + \kappa_i^2\ip{X}{\nu},
    \end{aligned} \]
    where we have used Codazzi equation \(\nabla_i h_{ij} = \nabla_j h_{ii}\). By differentiating the self-shrinker equation \eqref{eq:self-shrinker}, we have \(\nabla F = \frac{1}{\phi^2}\nabla \phi\). Then we have 
    \begin{equation}\label{eq:rigid}
        \begin{aligned}
            \sum_{i=1}^{n} F^{ii} \phi_{ii} &= \sum_i F^{ii}\sum_j \nabla_j h_{ii}\ip{X}{e_j} + \sum_i F^{ii}h_{ii} + \sum_i F^{ii}\kappa_i^2 \ip{X}{\nu} \\
            &= \nabla_j F \ip{X}{e_j} + F + \sum_i F^{ii}\kappa_i^2 \ip{X}{\nu} \\
            &=  \ip{X}{\nabla F} + F + \sum_i F^{ii}\kappa_i^2 \ip{X}{\nu} \\
            &= \frac{1}{\phi^2}\ip{X}{\nabla \phi} + \left( F - \frac{1}{F}\sum_i F^{ii}\kappa_i^2\right), 
        \end{aligned}
    \end{equation} 
    where we have used Euler's identity \(\sum_{i,j=1}^{n}F^{ij}h_{ij} = F\) in the second equality, and the equation \eqref{eq:self-shrinker} again in the last equality. Since \(F = s_k^{1/k}\), \(F^{ii} = \frac{1}{k}s_k^{1/k-1}s_k^{ii}\),  we have 
    \[ \frac{1}{F}\sum_i F^{ii}\kappa_i^2 = \frac{1}{k s_k}\sum_i s_k^{ii}\kappa_i^2. \]
    From the well-known identity  
    \[ \begin{aligned}
        \sum_{i=1}^n \sigma_k^{ii}(\kappa)\kappa_i^2 &= \sum_{i=1}^n \sigma_{k-1}(\kappa | i) \kappa_i^2 \\
        &= \sigma_1(\kappa)\sigma_k(\kappa) - (k+1)\sigma_{k+1}(\kappa),
    \end{aligned}  \]
    we deduce that 
    \[ \sum_{i=1}^{n} s_k^{ii}\kappa_i^2 = ns_1(\kappa)s_k(\kappa) - (n-k)s_{k+1}(\kappa). \]
    By Newton-Maclaurin's inequality, \(s_1 \geq s_k^{1/k}\), \( s_{k+1} \leq s_k^{(k+1)/k} \), and the equality holds if and only if all principal curvatures are equal to each other. We conclude that \(F - \frac{1}{F}\sum_i F^{ii}\kappa_i^2 \leq 0\) and then \(\sum_{i,j=1}^{n} F^{ij}\phi_{ij} \leq \frac{1}{\phi^2}\ip{X}{\nabla \phi}\) from \eqref{eq:rigid}. By strong maximum principle, \(\phi \) must be constant on \(M\). Then \(\nabla^2\phi = \nabla \phi \equiv 0\), which implies \(F - \frac{1}{F}\sum_i F^{ii}\kappa_i^2 \equiv 0\), then \(M\) is totally umbilical. Thus \(M\) is a branch of hyperboloid \(\H^n\). 
\qed 

\section{Evolution Equations}\label{sec:evolution}
In a local coordinate system \(\{x_1,\cdots, x_n\}\), the first and second fundamental forms are defined by 
\[ g_{ij} = \ip{X_i}{X_j}, \quad h_{ij} = -\ip{\nabla_i \nabla_j X}{\nu}. \]
Here \(\nabla_i \nabla_j X = X_{ij} - \Gamma_{ij}^k  X_k\) is the covariant derivative.
\par For a spacelike hypersurface \((M^n,g)\) in \(\R^{n,1}\), let \(\nabla\) be its Levi-Civita connection. The Riemannian curvature tensor is defined by
\begin{equation}\label{eq:Riem-def}
    {R}(X,Y)Z = -\nabla_X\nabla_Y Z + \nabla_Y\nabla_X Z + \nabla_{[X,Y]}Z. 
\end{equation}
By using a local frame $\{e_1,\cdots,e_n\}$, we define 
\begin{equation}
    {R}_{ijkl} = g(R(e_i,e_j)e_k,e_l), \quad {R}_{ikj}^{\quad p} = {R}_{ikjm}g^{mp}.
\end{equation}
Recall the following two fundamental geometric equations: 
\begin{equation}\label{eq:Gauss-eq}
    R_{ijkl} = -(h_{ik}h_{jl} - h_{il}h_{jk}) \quad \text{(Gauss equation)}.
\end{equation}
\begin{equation}\label{eq:Codazzi-eq}
    \nabla_k h_{ij} = \nabla_j h_{ik} \quad \text{(Codazzi equation)}.
\end{equation}

Let \(S(p,t):=\frac{1}{F}(p,t)\). We will denote  $\dot{F}^{ij}=\frac{\partial F}{\partial h_{ij}}$, $\ddot{F}^{ij,kl}=\frac{\partial^2 F}{\partial h_{ij} \partial h_{kl}}$
and $\dot{S}^{ij}=\frac{\partial S}{\partial h_{ij}}$, $\ddot{S}^{ij,kl}=\frac{\partial^2 S}{\partial h_{ij} \partial h_{kl}}$. We firstly derive evolution equations for some geometric quantities.
\begin{proposition}\label{prop:evolution_geo_quantaties}
    Under ICF \eqref{eq:ICF}, the following evolution equations hold:
    \begin{enumerate}[label=(\arabic*)]
        \item \(\displaystyle \pd{}{t}g_{ij} = -2Sh_{ij}\),
        \item \(\displaystyle \pd{}{t}\nu = -\nabla S\),
        \item \(\displaystyle \pd{}{t}h_{ij} = -\nabla_i \nabla_j S - SA_i^p A_j^q g_{pq}\),
        \item \(\pd{}{t}F = F^{-2}\dot{F}^{ij}\nabla_i\nabla_j F - 2F^{-3}\dot{F}^{ij}\nabla_i F \nabla_j F + F^{-1}\dot{F}(h^2)\). 
    \end{enumerate}
\end{proposition}
\begin{proof}
We follow the calculations in \cite{co-compact}. Using Gauss-Weingarten formula \eqref{eq:GW-formula}, the evolution of the induced metric \(g\) is given by:
\begin{equation}
    \begin{aligned}
        \pd{}{t}g_{ij} &= \pd{}{t}\ip{\pd{X}{x^i}}{\pd{X}{x^j}} \\
        &= \ip{\pd{}{x^i}(-S\nu)}{\pd{X}{x^j}} + \ip{\pd{X}{x^i}}{\pd{}{x^j}(-S\nu)} \\
        &= -SA_i^k g_{kj} - S g_{ik}A_j^k \\
        &= -2Sh_{ij}.
    \end{aligned}
\end{equation}
The evolution for the unit future directed normal \(\nu\) is given by (notice that \(\nu \cdot \nu = -1\)):  
\begin{equation*}
    \begin{aligned}
        \ip{\pd{\nu}{t}}{\pd{X}{x^i}} &= \pd{}{t}\ip{\nu}{\pd{X}{x^i}} - \ip{\nu}{\pd{}{x^i}(-S\nu)} \\
        &= \ip{\nu}{\pd{}{x^i}(S\nu)} \\
        &= \ip{\nu}{\pd{S}{x^i}\nu + SA_i^k \pd{X}{x^k}} \\
        &= -\pd{S}{x^i},
    \end{aligned}
\end{equation*}
which leads to 
\begin{equation}
    \pd{\nu}{t} = -\pd{S}{x^i} g^{ij} \pd{X}{x^j} = -\nabla S.
\end{equation}
The evolution for the second fundamental form \(h\) is
\begin{equation}
    \begin{aligned}
        \pd{}{t}h_{ij} &= \pd{}{t}\ip{\pd{\nu}{x^i}}{\pd{X}{x^j}} \\
        &= \ip{\pd{}{x^i}\left(\pd{\nu}{t}\right)}{\pd{X}{x^j}} + \ip{\pd{\nu}{x^i}}{\pd{}{x^j}\pd{X}{t}} \\
        &= \ip{\pd{}{x^i}\left(-\pd{S}{x^k}g^{kl}\pd{X}{x^l} \right)}{\pd{X}{x^j}} + \ip{\pd{\nu}{x^i}}{\pd{}{x^j}(-S\nu)}\\
        &= -\left(\frac{\partial^2 S}{\partial x^i \partial x^j} - \Gamma_{ij}^k \pd{S}{x^k} \right) - \ip{\pd{\nu}{x^i}}{S \pd{\nu}{x^j}} \\
        &= -\nabla_i\nabla_j S - SA_i^p A_j^q g_{pq}.
    \end{aligned}
\end{equation}

Introduce the canonical spacetime connection  (see \cite{Andrews-Ben-The-Ricci-Flow}, Section 6.3)
\begin{equation}\label{eq:canonical-spacetime-connection}
    \nabla_t \partial_i:= \sum_{p=1}^{n}\left(\frac{1}{2}g^{pj}\partial_t g_{ij}\right)\partial_p = -S\sum_{p=1}^{n}A_i^p \partial_p,
\end{equation}
where \(\partial_i := \pd{X}{x^i}\) (\(i=1,\cdots,n\)). Then \(\nabla_t g_{ij} = 0\), and we have:
\begin{equation}\label{eq:original-evolution of h}
    \nabla_t h_{ij} = \pd{}{t}h_{ij} - h(\nabla_t \partial_i,\partial_j) - h(\partial_i, \nabla_t \partial_j) = -\nabla_i\nabla_j S + SA_j^p h_{ip}.
\end{equation}
Since $\nabla_t g_{ij} = 0$, we have 
\begin{equation}\label{eq:original-evolution of S}
    \pd{}{t}S(h,g) = \dot{S}^{ij}\nabla_t h_{ij} = -\dot{S}^{ij}\nabla_i\nabla_j S + S\dot{S}(h^2),
\end{equation}
where \(\dot{S}(h^2) = \dot{S}^{ij}A_j^p h_{pi}\).
From Gauss-Codazzi equations and Ricci's identity,
\begin{equation}
    \begin{aligned}
        \nabla_i \nabla_j h_{kl} &= \nabla_i \nabla_k h_{jl} \\
        &= \nabla_k \nabla_i h_{jl} + R_{ikj}^{\quad p} h_{pl} + R_{ikl}^{\quad p} h_{jp} \\
        &= \nabla_k \nabla_l h_{ij} + h_{jk}h_i^p h_{pl} - h_{ij}h_k^p h_{pl} + h_{kl}h_i^p h_{jp} - h_{il}h_k^p h_{jp}.
    \end{aligned}
\end{equation}
Plug it into \eqref{eq:original-evolution of h}, then we get:
\begin{equation}
    \begin{aligned}\label{eq:eoh}
        \nabla_t h_{ij}&= -\nabla_i(\dot{S}^{kl}\nabla_j h_{kl}) + SA_j^p h_{ip} \\
        &= -\dot{S}^{kl}\nabla_i\nabla_j h_{kl} - \ddot{S}^{kl,mn}\nabla_i h_{mn}\nabla_j h_{kl} + SA_j^p h_{ip} \\
        &= -\dot{S}^{kl}(\nabla_k \nabla_l h_{ij} - h_{ij}h_k^p h_{pl} + h_{kl}h_i^p h_{jp}) - \ddot{S}(\nabla_i h, \nabla_j h) + SA_j^p h_{ip} \\
        &= -\dot{S}^{kl}\nabla_k \nabla_l h_{ij} - \ddot{S}(\nabla_i h, \nabla_j h) + (S - \dot{S}(h))h_i^p h_{jp} + h_{ij}\dot{S}(h^2),
    \end{aligned}
\end{equation}
where \(\dot{S}(h) = \dot{S}^{kl}h_{kl},\dot{S}(h^2) = \dot{S}^{ij}h_j^p h_{pi},\) and \(\ddot{S}(\nabla_i h, \nabla_j h) = \ddot{S}^{kl,mn}\nabla_i h_{mn}\nabla_j h_{kl}\). Since \(S = F^{-1}\), 
\[\dot{S}^{kl} = -F^{-2}\dot{F}^{kl}, \quad \ddot{S}^{kl,mn} = 2F^{-3}\dot{F}^{mn}\dot{F}^{kl}-F^{-2}\ddot{F}^{kl,mn}, \]
then \eqref{eq:eoh} can be expressed
\begin{equation}\label{eq:nabla-t-h}
    \begin{aligned}
        \nabla_t h_{ij} = &F^{-2}\dot{F}^{kl}\nabla_k\nabla_l h_{ij}+F^{-2}\ddot{F}(\nabla_i h, \nabla_j h)-2F^{-3}\nabla_i F \nabla_j F\\
        &+ 2F^{-1}h_i^p h_{jp} - F^{-2}h_{ij}\dot{F}(h^2),
    \end{aligned}
\end{equation}
where we have used Euler's identity \(\sum_{k,l=1}^{n}\dot{F}^{kl}h_{kl} = F\). 
\par From \eqref{eq:nabla-t-h} we deduce 
\begin{equation}\label{eq:evolution of h}
    \pd{}{t} h_{ij} = F^{-2}\dot{F}^{kl}\nabla_k\nabla_l h_{ij}+F^{-2}\ddot{F}(\nabla_i h, \nabla_j h)-2F^{-3}\nabla_i F \nabla_j F - F^{-2}h_{ij}\dot{F}(h^2).
\end{equation}

\end{proof}

\section{Gauss Map Parameterization}\label{sec:gauss-map}

For a strictly convex spacelike hypersurface \(M \), define its support function \(u:\H^n \to \R\) by \[ u(z) = \inf\{-\ip{z}{X}: X \in M\}. \]  
If \(M\) is co-compact, strictly convex, \(u(z)\) is bounded on \( \H^n\), and is written as:
\[ u(z) = -\ip{z}{X(\nu^{-1}(z))},  \]
where we regard \(z\in \H^n\) as the position vector of the future branch of hyperboloid. \(X(\nu^{-1}(z))\) means the position vector of \(M\) whose unit future directed normal is exactly \(z\). 
\begin{remark}
  For a co-compact hypersurface, by a translation along the \(x_{n+1}\)-direction, we can assume that the support function \(u\) is positive for all \(z \in \H^n\). 
\end{remark}

\par Under the flow \eqref{eq:ICF}, we will derive the evolution equation for \(u = u(z,t)\). Differentiating $u$ with respect to \(t\), we have: 
\begin{equation}\label{eq:pd{u}{t}}
    \pd{u}{t}(z,t) = \ip{-\pd{X}{t}\big|_{\nu^{-1}(z)} - X_{*}\left( \pd{}{t}\nu^{-1}(z) \right)}{z} = -S\big|_{\nu^{-1}(z)}.
\end{equation}
Here $X_*$ is the push-forward of $X$.

If we use the image of Gauss map to reparameterize the hypersurface \(M\), then the position vector of \(M\) can be recovered by its support function \cite{co-compact}:
\begin{equation}
    \bar{X} = uz - z_{*}(\bnabla u) \in \R^{n,1},
\end{equation}
where $\bnabla$ denotes the Levi-Civita connection on \(\H^n\) with respect to the standard hyperbolic metric \(\bg\) (the induced metric of hyperboloid from \(\R^{n,1}\)), and $z_*$ is the push-forward of the map $z:\H^n\rightarrow \R^{n,1}$.  
Moreover, the eigenvalues \( \lambda=(\lambda_1,\dots,\lambda_n)\) of the following tensor
\begin{equation}
    \tau_{ij} := u \bg_{ij} - \bnabla_i \bnabla_j u
\end{equation}
with respect  to \(\bg\) are exactly the principle radii of \(M\) (the reciprocal  of \(\kappa_i\)), and it satisfies the Codazzi-type equality \(\bnabla_i \tau_{jk} = \bnabla_{j}\tau_{ik}\) following from Ricci's identity.

\par Since we assume that $M$ is strictly convex, then  \(\tau_{ij}\) is positive definite. Therefore, we can define the dual function \(F_{*}\) of \(F\): 
\[F_{*}(\tau,\bar{g}) := F(\bar{g},\tau)^{-1}=\left(\frac{s_n}{s_{n-k}}\right)^{1/k}(\lambda).\]

Since \(S(z,t) = F(h,g)^{-1} = F_{*}(\tau, \bg)\) by using the Gauss map parameterization, the evolution equation of \(u\) can be written as: 
\begin{equation}\label{eq:evolution of support function}
    \pd{u}{t}(z,t) = -F_{*}(\tau)(z,t),
\end{equation}
where we omitted \(\bg\) in \(F_*(\tau, \bg)\) since \(\bg\) does not change under the flow, and we may also omit \(\tau\) in the following context. 
\par Denote \(\dot{F}_*^{kl} = \pd{F_*}{\tau_{kl}}\). The evolution equation \eqref{eq:evolution of support function} can be rewritten as
\begin{equation}\label{eq:expanded form of evolution of support function}
    \begin{aligned}
        \pd{}{t}u &=-F_* = -\dot{F}_*^{kl}\tau_{kl}\text{ (since \(\operatorname{degree}F_* = 1\)) } \\
        &= \dot{F}_*^{kl}\bnabla_k \bnabla_l u - u\dot{F}_*^{kl}\bg_{kl}\text{ (by the definition of \(\tau_{kl}\))}.
    \end{aligned}
\end{equation}
\begin{remark}\label{rem:induced functions}
    Under the co-compactness assumption, the maximum and minimum of \(\kappa_i\) can be attained. In addition, \(F_*\) and \(u\) have positive upper and lower bounds at every moment. 
\end{remark}
\begin{proposition}\label{prop:evolution of speed}
    The evolution for $F_* = F_*(\tau)$ under ICF \eqref{eq:ICF} is given by: 
    \begin{enumerate}[label=(\arabic*)]
        \item $\displaystyle \pd{}{t}F_{*} = \dot{F}_*^{kl}\bnabla_k \bnabla_l F_* - F_* \dot{F}_*^{kl}{\bg}_{kl}$,
        \item $\displaystyle \pd{}{t}F_{*}^{-1} = F_*^{-1} \dot{F}_*^{kl}\bg_{kl} - 2F_* \dot{F}_*^{kl}\bnabla_k F_*^{-1}\bnabla_l F_*^{-1} + \dot{F}_*^{kl}\bnabla_k \bnabla_l F_*^{-1} $.
    \end{enumerate}
    where \(\dot{F}_*^{kl}:= \pd{F_*}{\tau_{kl}}\) is the partial derivative of \(F_*\) (as a function with the matrix components \(\tau_{ij}\) as its variables).  
\end{proposition}
\begin{proof}
    From \eqref{eq:expanded form of evolution of support function} we get 
    \[ \begin{aligned}
        \pd{}{t}F_* &= \dot{F}_*^{kl}\pd{}{t}\tau_{kl} \\
        &= \dot{F}_*^{kl}\pd{}{t}\left( u\bg_{kl} - \bnabla_k\bnabla_l u \right) \\
        &= \dot{F}_*^{kl}\left(\pd{u}{t}\bg_{kl} - \bnabla_k \bnabla_l \pd{u}{t}\right) \\
        &= \dot{F}_*^{kl}\bnabla_k \bnabla_l F_* - F_* \dot{F}_*^{kl}\bg_{kl}.
    \end{aligned}  \]
    The second equation can be easily obtained by the first equation.
\end{proof}

\section{A Priori Estimates for the Unnormalized Flow}\label{sec:unnormalized-flow}

\subsection{Short-Time Existence}
Since the equation \eqref{eq:ICF} is invariant for \(G\) action, the flow equation actually deforms the quotient manifold \(M/G\). Thus we have:  

\begin{proposition}
    The ICF \eqref{eq:ICF} admits a short-time solution \(M_t, t\in [0,T)\) which is co-compact.  
\end{proposition}
For details, see section 5 of \cite{co-compact} for reference.

\subsection{Estimates on Support Function}
Let \(u(z,t)\) be the support function of a standard co-compact, strictly convex spacelike hypersurface \(X:M\times [0,T) \to\R^{n,1}\). Then \(u\) has positive lower and upper bounds at initial time \(t = 0\). Therefore, we obtain the following estimate.
\begin{proposition}\label{prop:support-estimate}
    Let \(X_0:M\to \R^{n,1}\) be standard co-compact, then the support function \(u(z,t)\) satisfies:
    \begin{equation}\label{eq:std supp func est}
        e^{-t}\inf_{z\in \H^n}u(z,0) \leq u(z,t) \leq e^{-t}\sup_{z\in \H^n}u(z,0), \quad \forall (z,t)\in \H^n \times [0,T).
    \end{equation}
\end{proposition}
\begin{proof}
    Let \(w(t) = u_+ e^{-t}\), where \(u_+ = \sup_{z\in \H^n} u(z,0)\) is a finite positive constant. Then \(\frac{\dd }{\dd t}w(t) = -w(t)\). 
    \[ \begin{aligned}
        \pd{}{t}(u - w) &= -F_*(\tau) + w \\
        &= -F_*(\tau) + F_*(w\bg) \\
        &= \int_0^1 \frac{\dd}{\dd s}F_*((1-s)\tau + sw\bg)\dd s \\
        &= a^{kl}(-\tau_{kl} + w\bg_{kl}) \\
        &= a^{kl}\left(\bnabla_k \bnabla_l(u - w)-\bg_{kl}(u - w) \right),
    \end{aligned}
    \]
    where \(a^{kl} = \int_0^1 F_*^{kl}\big|_{(1-s)\tau+ sw\bg} \dd s\) is symmetric and positive-definite. Since at the initial time \(u(z,0)-w(0)\leq 0\), by the maximum principle of parabolic equations, 
    \[ u(z,t) - w(t) \leq 0, \forall z\in \H^n, \forall t\in[0,T).\] 
    That is, \(u(z,t)\leq u_+ e^{-t}\). Similarly, we can prove the lower bound of $u$.
    \end{proof}

\subsection{Estimates on Speed Function}
We establish the \(C^0\)-estimate on speed function \(F_*\).
\begin{proposition}\label{prop:speed-estimate}
    Suppose that the initial hypersurface \(X_0:M\to \R^{n,1}\) is standard co-compact, then \(F_*(z,t)\) satisfies: 
    \begin{equation}
        \inf_{z\in \H^n}u F_*^{-1}(z,0) \leq u F_*^{-1}(z,t) \leq \sup_{z\in \H^n}u F_*^{-1}(z,0).
    \end{equation}
   Therefore, we have two positive constants \(C_-, C_+ \), such that  
    \begin{equation}\label{eq:std speed func est}
        C_- e^{-t} \leq F_*(\tau)(z,t) \leq C_+ e^{-t}.
    \end{equation}
\end{proposition}
\begin{proof}
    Let \(w(z,t) = uF_*^{-1}\), by \eqref{eq:expanded form of evolution of support function} and Proposition ~\ref{prop:evolution of speed}, then we have 
    \begin{equation*}
        \begin{aligned}
            \pd{}{t}w &= \pd{u}{t}F_*^{-1} + u\pd{}{t}F_*^{-1} \\
            &= \dot{F}_*^{kl}\bnabla_k \bnabla_l w - 2F_* \dot{F}_*^{kl}\bnabla_k F_*^{-1}\bnabla_l w.
        \end{aligned}
    \end{equation*}
    Since \(\dot{F}_*^{kl}\) is positive definite, by maximum principle, we have \(\inf_{z\in \H^n}w(z,0) \leq w(z,t)\leq \sup_{z\in \H^n}w(z,0)\). By co-compactness, \(w(z,0)\) has positive lower and upper bounds. Thus \(C_1\leq w(z,t) \leq C_2\). Then by \eqref{eq:std supp func est}, we get \eqref{eq:std speed func est}.
\end{proof}

\subsection{Gradient Estimate}
We establish an upper bound for \(\abs{\bnabla u}\).
\begin{proposition}\label{prop:gradient-estimate}
    Suppose that \(X:M\to \R^{n,1}\) is spacelike, strictly convex, and standard co-compact, then its support function \(u(z)\) satisfies:
    \begin{equation}\label{eq:gradient-estimate}
        \abs{\bnabla u}^2 \leq u^2.
    \end{equation}
\end{proposition}
\begin{proof}
    Consider \( \phi(z) = \ip{X}{X} = \abs{\bnabla u}^2 - u^2 \). Our goal is to prove \(\phi \leq 0\). Due to the cocompactness, \(\phi\) achieves its maximum at some \(y_0 \in \H^n\). Choose an normal coordinate \((z_1,\cdots,z_n)\) around \(y_0\). We have, at \(y_0\):  
    \[ 0 = \phi_i = 2\ip{X_i}{X}, \quad \forall i=1,\cdots,n.  \]
    By strict convexity of the hypersurface, \(\{X_i\}_{i=1}^n\) are linearly independent, forming a basis for \(T_p M\), where \(p = \nu^{-1}(y_0)\in M\). Then we conclude that \(X\) is parallel to its normal direction \(\nu\) at \(y_0\). Since \(X = \bnabla u - uz\), we have \(\bnabla u (y_0) = 0\). Therefore, \(\phi \leq \phi(y_0) = -u^2(y_0)\leq 0\).

\end{proof}

\subsection{Estimates on Principle Curvatures}
We firstly derive the evolution equation for \(\tau_{ij}\). We will denote \(\dot{F}_* (\bg) := \dot{F}_*^{kl}\bg_{kl}\). 
\begin{proposition}
    Under the ICF \eqref{eq:ICF}, there holds:
    \begin{equation}\label{eq:evolution of radii}
        \pd{}{t}\tau_{ij} = \dot{F}_*^{kl}\bnabla_k \bnabla_l \tau_{ij} + \ddot{F}_*^{kl,mn}\bnabla_i \tau_{kl}\bnabla_j \tau_{mn} - 2F_* \bg_{ij} + \dot{F}_*(\bg)\tau_{ij}.
    \end{equation}
\end{proposition}
\begin{proof}
    By definition, \(\tau_{ij} = u\bg_{ij} - \bnabla_i \bnabla_j u\). Then from \eqref{eq:evolution of support function} we get: 
    \[ \begin{aligned}
        \pd{}{t}\tau_{ij} &= \pd{u}{t}\bg_{ij} - \bnabla_i \bnabla_j \pd{u}{t} \\
        &= -F_* \bg_{ij} + \bnabla_i \bnabla_j F_* \\
        &= -F_* \bg_{ij} + \bnabla_i \left(\dot{F}_*^{kl}\bnabla_j \tau_{kl} \right) \\
        &= \dot{F}_*^{kl}\bnabla_i \bnabla_j \tau_{kl} + \ddot{F}_*^{kl,mn}\bnabla_i \tau_{kl} \bnabla_j \tau_{mn} - F_* \bg_{ij}.
    \end{aligned}
    \]
     Set \(\bar{R}(X,Y)Z = -\bnabla_X\bnabla_Y Z + \bnabla_Y\bnabla_X Z + \bnabla_{[X,Y]}Z\), and \(\bar{R}_{ijkl} = \bg(R(e_i,e_j)e_k,e_l)\), \( \bar{R}_{ikj}^{\quad p} = \bar{R}_{ikjm}g^{mp} \). By Ricci's identity, we have
    \[ \begin{aligned}
        \bnabla_i \bnabla_j \tau_{kl} &= \bnabla_i \bnabla_k \tau_{jl} \\
        &= \bnabla_k \bnabla_i \tau_{jl} + \bar{R}_{ikj}^{\quad p}\tau_{pl} + \bar{R}_{ikl}^{\quad p}\tau_{jp}\\
        &= \bnabla_k \bnabla_l \tau_{ij} + \bg_{kj}\tau_{il} - \bg_{ij}\tau_{kl} + \bg_{kl}\tau_{ij} - \bg_{il}\tau_{kj},
    \end{aligned} \]
    Using the above equality and the Euler's identity \(\sum_{k,l=1}^{n}\dot{F}_*^{kl}\tau_{kl} = F_*\), we have 
    \[ \pd{}{t}\tau_{ij} = \dot{F}_*^{kl}\bnabla_k \bnabla_l \tau_{ij} + \ddot{F}_*^{kl,mn}\bnabla_i \tau_{kl}\bnabla_j \tau_{mn} - 2F_* \bg_{ij} + \dot{F}_*^{kl}\bg_{kl}\tau_{ij}. \]
\end{proof}

\begin{remark}\label{rem:curvature-estimate}
   It seems that the idea using in  \cite{co-compact} can not  establish uniform curvature estimates of the inverse curvature flow. To overcome this difficulty, we will instead establish the uniform curvature estimates to the normalized flow in Section \ref{sec:normalized-flow}. As we will demonstrate in Lemma \ref{lem:estimate-on-eigenvalues}, a uniform upper bound for the principal radii of the normalized flow can be established independently, without relying on any prior bounds from the unnormalized flow. Therefore, the scaling relation between the normalized and unnormalized flows will naturally yield an exponential decay bound $\tilde{C}e^{-t}\bg_{ij} \leq \tau_{ij} \leq C e^{-t}\bg_{ij}$ for the unnormalized flow.
\end{remark}

We conclude the above remark to the following lemma. The crucial curvature estimate (Lemma \ref{lem:estimate-on-eigenvalues}) will be proved later in Section \ref{sec:normalized-flow}.
\begin{lemma}
    There exists tow time-independent constant \(C,\tilde{C} > 0\) (only depending on \(X_0\)), such that \(\tau_{ij}\) has the upper bound:
    \begin{equation}\label{eq:max radii estimate}
       \tilde{C}e^{-t}\bg_{ij}\leq  \tau_{ij} \leq C e^{-t}\bg_{ij}.
    \end{equation}
\end{lemma}
\begin{proof}
    By Lemma \ref{lem:estimate-on-eigenvalues} and the scaling relation between the normalized and unnormalized flows.
\end{proof}

\section{Normalized Flow and Long-time Existence}\label{sec:normalized-flow}
In the next two sections, we mainly work on the normalized hypersurface: 
\begin{equation}
    \tilde{X}(p,t) := e^t X(p,t).
\end{equation}
If \(X(p,t)\) satisfies the unnormalized flow \eqref{eq:ICF}, then \(\tilde{X}(p,t)\) satisfies the normalized flow:  
\[ \pd{\tilde{u}}{t}(y,t) = -F_*(\tilde{\tau})(y,t) + \tilde{u}(y,t), \quad \forall y\in \H^n, t\in (0,T). \]
Here \(\tilde{u}\)  is the support function of \(\tilde{X}\) and $\tilde{\tau}=\tilde{u}\overline{g}-\overline{\nabla}^2\tilde{u}$
is the curvature radius tensor of $\tilde{X}$. Similarly, we use tilde to denote all geometric quantities related to the normalized hypersurface \(\tilde{X}\). However, to simplify the notations, we may drop the tildes when there is no ambiguity, and simply write the normalized flow as 
\begin{equation}\label{eq:normalized-flow} 
    \pd{u}{t}(y,t) = -F_*(\tau)(y,t) + u(y,t), \quad \forall y\in \H^n, t\in (0,T). 
\end{equation}

Let \(G:M\to \H^n\) be the Gauss map of the strictly convex cocompact hypersurface \(M\) in the ambient Minkowski space \(\R^{n,1}\). Denote \(\mathscr{L} = \partial_t - F_*^{kl}\bnabla_k \bnabla_l\) with respect to the hyperbolic metric \(\bg\). 
Let \(Z := \dot{F_*}(\bg)-1 \). By the concavity and the homogeneity we know that \(Z \geq 0\), and \(Z=0\) if and only if \(\kappa_1=\kappa_2 =\cdots=\kappa_n\) (see Lemma \ref{lem:algebraic}).

We first derive some evolution equations for the normalized flow. Here we omit the tildes for the geometric quantities of the normalized flow. 
\begin{lemma}\label{lem:normalized-equation-for-u}
The normalized evolution equation for the support function \(u\) is
\[ \mathscr{L}u = \partial_t u - F_*^{kl}\bnabla_k\bnabla_l u = - Zu.\]
\end{lemma}

\begin{lemma}
    The normalized evolution equation for speed $F_*$ is 
    \[\mathscr{L}F_* =  \partial_t F_* - F_*^{kl}\bnabla_k\bnabla_l F_* = - F_* \left(\dot{F_*}(\bg) - 1\right) = - ZF_*.  \]
\end{lemma}
\begin{lemma}
    The normalized evolution equation for the curvature radius tensor \(\tau_{ij}\) is given by 
    \[ \mathscr{L}\tau_{ij} = \partial_t \tau_{ij} - F_*^{kl}\bnabla_k\bnabla_l \tau_{ij} = F_*^{kl,mn}\bnabla_i\tau_{kl}\bnabla_j \tau_{mn} - 2F_* \bg_{ij} + \dot{F_*}(\bg)\tau_{ij} + \tau_{ij}.  \] 
\end{lemma}
The above equations can be easily verified by using \(\tilde{u} = e^t u\), \(\tilde{F_*} = e^t F_*\) and \(\tilde{\tau}_{ij} = e^t \tau_{ij}\). Moreover, by Proposition ~\ref{prop:support-estimate}, Proposition ~\ref{prop:speed-estimate}, and Proposition ~\ref{prop:gradient-estimate}, we have:
\begin{proposition}\label{prop:normalized-estimates}
    The normalized equation \eqref{eq:normalized-flow} satisfies 
    \begin{equation}
        \begin{aligned}
            &C_1 \leq \tilde{u}(y,t) \leq C_2, \\
            &C_1\leq \tilde{F}_*(y,t) \leq C_2, \\
            &\abs{\bnabla \tilde{u}}^2(y,t) \leq \tilde{u}^2(y,t),
        \end{aligned}\quad \forall (y,t)\in \H^n\times (0,T)
    \end{equation}
    for some positive constants \(C_1,C_2\) independent of \(T\).
\end{proposition}
\subsection{Estimates on Principal Curvature}
Now we are going to bound \(\tilde{\tau}_{ij}\) from above. We will consider the test function 
\[ \varphi = {\tilde{u}^\alpha}\cdot{\tilde{\lambda}_1}, \]
where \(\tilde{\lambda}_1\) is the largest eigenvalue of \(\tilde{\tau}\), and \(\alpha> 0\) is a positive constant to be determined. 
\par We will first extend the lemma 5 in \cite{Brendle2017} to the parabolic case:
\begin{lemma}\label{lem:Brendle-lem}
    Suppose that \(h = h(t)\) is a time-dependent symmetric \((0,2)\)-tensor on a manifold \(M\) with a time-dependent metric \(g(t)\). Let \(\nabla\) be the Levi-Civita connection of \(g(t)\). Suppose the largest eigenvalue \(\lambda_1\) of \(h\) (with respect to the metric \(g\)) has multiplicity \(m\) at \((y_0,t_0)\in M \times \R\) and the eigenvalues are arranged in descending order that $\lambda_1 = \cdots = \lambda_m > \lambda_{m+1} \geq \cdots \geq \lambda_n$. 
    Choose a local normal coordinate system $(y^1, \dots, y^n)$ around $y_0$ such that at $(y_0, t_0)$, the following hold: 
    \begin{itemize}
        \item $g_{ij} = \delta_{ij}$; 
        \item $h$ is diagonalized, and $h_{ij} = \lambda_i \delta_{ij}$;
    \end{itemize}
    Then for any smooth function \(\varphi(y,t)\) such that \(\varphi \geq \lambda_1\) in \(M \times \R\), and \(\varphi = \lambda_1\) at \((y_0,t_0)\), we have, at \((y_0,t_0)\), \(\forall i=1,\cdots,n\),
    \begin{equation}\label{eq:Brendle-lem}
        \left\{ \begin{aligned}
                \delta_{\alpha\beta}\nabla_i \varphi = \nabla_i h_{\alpha\beta}, \quad 1\leq \alpha,\beta\leq m \\
                \nabla_i \nabla_i \varphi \geq \nabla_i\nabla_i h_{11} + 2\sum_{p>m}\frac{h_{1pi}^2}{\lambda_1 - \lambda_p}
            \end{aligned}\right. 
    \end{equation}
    Moreover, at \((y_0,t_0)\), we have 
    \begin{equation}\label{eq:Brendle-lem-t}
        \delta_{\alpha\beta}\pd{\varphi}{t} = \nabla_{t}h_{\alpha\beta}, \quad 1\leq \alpha,\beta\leq m,
    \end{equation}
    where \(\nabla_t\) is defined by 
    \begin{equation}
    \nabla_t \partial_i:= \sum_{p=1}^{n}\left(\frac{1}{2}g^{pj}\partial_t g_{ij}\right)\partial_p.    
    \end{equation}  
Note that, if $h$ is the second fundamental form, the above connection is  the canonical spacetime connection defined by \eqref{eq:canonical-spacetime-connection}.
 
 \end{lemma}
\begin{proof}
    The first assertion \eqref{eq:Brendle-lem} follows directly from the lemma 5 in \cite{Brendle2017}. We only need to prove \eqref{eq:Brendle-lem-t}. Let \((y^1,\cdots,y^n,t)\) be a local coordinate of \(M\times \R\) around \((y_0,t_0)\) such that \(g\) and \(h\) are diagonal at \((y_0,t_0)\), and \(g_{ij}=\delta_{ij}\) at \((y_0,t_0)\). By switching the indexes, we also assume that \(\lambda_1 = h_{11} = \cdots = h_{mm}\) at $(y_0,t_0)$. Let \(\gamma(t) = (y_0,t), t\in (t_0-\epsilon,t_0 + \epsilon)\) be a curve passing through \((y_0,t_0)\), with \(\dot{\gamma}(t) = \partial_t\). Let \[ v(t) = \sum_{i=1}^{n}v^i(t)\pd{}{y^i}(y_0,t)\] be a spatial vector field along \(\gamma(t)\), such that \(v(t_0) \in \text{span}(e_1,\cdots,e_m)\) and \(v'(t_0) \in \text{span}(e_{m+1},\cdots,e_n)\). That is, 
    \[ \begin{aligned}
        v^k(t_0) &= 0, \quad m < k \leq n, \\
        \pd{v^\alpha}{t}(t_0) &= 0, \quad 1\leq \alpha \leq m.
    \end{aligned} \]
    Since 
    \[ t\mapsto h(v(t),v(t)) - \varphi(\gamma(t))\cdot g(v(t),v(t)) \]
    attains its maximum at \(t = t_0\), we have 
    \[ \begin{aligned}
    0 &= \frac{\dd}{\dd t}\bigg|_{t=t_0}\left[h(v(t),v(t)) - \varphi(y_0,t)\cdot g(v(t),v(t))\right] \\
    &= \left(\pd{}{t}h_{ij}\right)v^i v^j + 2h_{ij}\pd{v^i}{t}v^j - \pd{\varphi}{t}g_{ij}v^i v^j \\
    &\quad -\varphi\left(\pd{}{t}g_{ij}\right)v^i v^j - 2\varphi g_{ij}\pd{v^i}{t} v^j \bigg|_{t=t_0}\\
    &= \left(\pd{}{t}h_{ij}\right)v^i(t_0) v^j(t_0) - \pd{\varphi}{t}g_{ij}v^i(t_0) v^j(t_0) -\varphi\left(\pd{}{t}g_{ij}\right)v^i(t_0) v^j(t_0).
    \end{aligned}
    \]
    Combined with the arbitrary choice of \(v(t_0) = \sum_{\alpha=1}^{m}v^\alpha(t_0) \pd{}{y^\alpha} \in \text{span}(e_1,\cdots,e_m)\) and the symmetry of \(\pd{}{t}h_{ij}\), \(g_{ij}\), and \(\pd{}{t}g_{ij}\), we conclude that 
    \[  \pd{}{t}h_{\alpha\beta} - \pd{\varphi}{t}g_{\alpha\beta} - \varphi \pd{}{t}g_{\alpha\beta} = 0,\quad \forall 1\leq \alpha,\beta \leq m.  \] 
  Thus, we have 
    \[ \begin{aligned}
        \nabla_t h_{\alpha\beta} &= \pd{}{t} h_{\alpha\beta} - h\left(\nabla_t \pd{}{y^\alpha}, \pd{}{y^\beta}\right) - h\left(\pd{}{y^\alpha}, \nabla_t\pd{}{y^\beta}\right)\\
        &= \pd{}{t} h_{\alpha\beta} - h\left(\frac{1}{2}g^{kj}\partial_{t}g_{\alpha j}\pd{}{y^k}, \pd{}{y^\beta}\right)-h\left( \pd{}{y^\alpha}, \frac{1}{2}g^{kj}\partial_{t}g_{\beta j}\pd{}{y^k},\right)\\
        &= \pd{}{t} h_{\alpha\beta} - \frac{h_{\alpha\alpha}+h_{\beta\beta}}{2}\partial_{t}g_{\alpha\beta} \\
        &= \pd{}{t} h_{\alpha\beta} - \varphi\partial_{t}g_{\alpha\beta}.
    \end{aligned}  \] 
    Therefore, \(\delta_{\alpha\beta}\pd{\varphi}{t} = \nabla_t h_{\alpha\beta}, 1\leq \alpha,\beta \leq m\).
\end{proof}

We will also exploit the following algebraic lemma. Note that functions of the form \((\sigma_k / \sigma_l)^{\frac{1}{k-l}}\) (for \(k > l \geq 0\)) satisfy the two conditions below.

\begin{lemma}\label{lem:algebraic}
    Let \(f(\lambda)\) be a \(C^2\)-function defined on a convex domain \(\Gamma \subset \R^n\) and \(e = (1,1,\cdots,1) \in \Gamma\). Let \(\Sigma \subset \mathbb{S}^{n-1}\) be a relatively compact subset of the unit sphere. Define \(Z(\lambda) := \sum_{i=1}^{n}\pd{f}{\lambda_i}(\lambda) - f(e)\). Suppose that \(f\) satisfies the following conditions: 
    \begin{enumerate}
        \item \(f\) is concave. Moreover, the Hessian matrix of \(f\) is negative-definite at any \(\lambda\in \Gamma\cap\{ e+t\cdot\Sigma \mid t > 0\}\), where \(\{ e+t\cdot\Sigma \mid t > 0\} := \{ e+t\cdot\sigma \mid t > 0, \sigma\in \Sigma\}\); 
        \item \(f\) is homogeneous of degree \(1\).
    \end{enumerate}
    Then, for any \(\eps > 0\), \(K_{\eps} \subset \Gamma\cap\{ e+t\cdot\Sigma \mid t\geq \eps\}\), we have 
    \[ \inf_{\lambda \in K_{\eps}}Z(\lambda) := \delta_{\eps} > 0. \]
\end{lemma}
\begin{proof}
    Fix any \(\sigma\in \Sigma\). Consider 
    \[ \mathcal{E}_{\sigma}(t) := f(e+t\sigma) - t\ip{\sigma}{D f(e+t\sigma)} - f(e). \]
    Then direct calculations and the concavity condition yield
    \[ \begin{aligned}
        \frac{d}{dt}\mathcal{E}_{\sigma}(t) &= \ip{D f(e+t\sigma)}{\sigma} - \ip{\sigma}{Df(e+t\sigma)} - t\ip{\sigma}{D^2 f(e+t\sigma)\cdot \sigma} \\
        &= -t D^2 f(e+t\sigma)(\sigma,\sigma) > 0,\qquad \forall t>0.
    \end{aligned} \]
    This means \(\mathcal{E}_{\sigma}(t)\) is strictly increasing for \(t > 0\). Hence 
    \[ \mathcal{E}_{\sigma}(t) \geq \mathcal{E}_{\sigma}(\eps) > \mathcal{E}_{\sigma}(0) = 0,\qquad \forall t\geq \eps. \] 
    Now by the continuity and the compactness of \(\Sigma\), we obtain 
    \[ \inf_{\sigma\in \Sigma; t\geq \eps} \mathcal{E}_{\sigma}(t) \geq \delta_{\eps} > 0. \]
    On the other hand, for any \(\lambda\in K_{\eps}\), there exists \(t\geq \eps, \sigma\in \Sigma\) such that \(\lambda = e+t\sigma\). For such a choice, we have 
    \[ \begin{aligned}
        \delta_{\eps} \leq \mathcal{E}_{\sigma}(t) &= f(\lambda) - \ip{\lambda - e}{Df(\lambda)} - f(e) \\
        &= f(\lambda) - \sum_{i=1}^{n}\lambda_i\cdot \pd{f}{\lambda_i}(\lambda) + \sum_{i=1}^{n}\pd{f}{\lambda_i} - f(e) \\
        &= \sum_{i=1}^{n}\pd{f}{\lambda_i} - f(e) = Z(\lambda).
    \end{aligned}  \]
    Here we used the Euler's identity for homogeneous functions in the last equality. This completes the proof.
\end{proof}

Now we are ready to prove the crucial curvature estimate.
\begin{lemma}\label{lem:estimate-on-eigenvalues}
    There exists a large constant \(\alpha > 1\) (depending only on \(\sup F_*\)) such that the normalized flow \eqref{eq:normalized-flow} satisfies
    \[ \tilde{u}^\alpha \cdot \tilde{\lambda_1} \leq C, \quad  \forall y\in \H^n, t\geq 0, \]
    where \(C>0\) is a constant depending only on \(n,\alpha, \inf_{\H^n\times\{0\}}u, \sup_{\H^n\times\{0\}}u\), and \( \sup_{\H^n\times\{0\}}\abs{\bnabla u}\).
\end{lemma} 
\begin{proof} 
    Throughout this proof, we may drop the tildes for simpler notations. Consider the test function \[ \phi = \alpha \log u + \log \lambda_1. \] 
    At the maximum point \((y_0,t_0)\) of \(\phi\) in \(\H^n \times [0,T]\), suppose that \(t_0 > 0\). Choosing a normal coordinate around \(y_0\), by Lemma \ref{lem:Brendle-lem}, we have at \((y_0,t_0)\),
    \[ 0 = \phi_i = \alpha\frac{u_i}{u} + \frac{(\lambda_1)_i}{\lambda_1},  \]
    \[ \begin{aligned}
        \phi_{ii} &= \alpha\frac{u_{ii}}{u} - \alpha\frac{u_i^2}{u^2} + \frac{(\lambda_1)_{ii}}{\lambda_1} - \frac{(\lambda_1)_i^2}{\lambda_1^2} \\
        &\geq \alpha\frac{u_{ii}}{u} - \alpha\frac{u_i^2}{u^2} + \frac{1}{\lambda_1}\left( \tau_{11ii} + 2\sum_{p>m}\frac{\tau_{1pi}^2}{\lambda_1 - \lambda_p} \right) - \frac{\tau_{11i}^2}{\lambda_1^2}.
    \end{aligned} \]
    Therefore, we get 
    \[ \begin{aligned}
            0 &\leq \mathscr{L}\phi \leq \frac{\mathscr{L}\tau_{11}}{\lambda_1} + \frac{\alpha\mathscr{L}u}{u} + \alpha\sum_i \frac{F_*^{ii}u_i^2}{u^2} - 2\sum_i \sum_{p>m}\frac{F_*^{ii}\tau_{1pi}^2}{\lambda_1(\lambda_1 - \lambda_p)} + \sum_i \frac{F_*^{ii}\tau_{11i}^2}{\lambda_1^2} \\
            &= \frac{\ddot{F_*}(\bnabla_1 \tau, \bnabla_1 \tau)}{\lambda_1} - \frac{2F_*}{\lambda_1} + (Z+2) - \alpha Z  \\
            &\quad + \alpha \sum_i \frac{F_*^{ii}u_i^2}{u^2} - 2\sum_i\sum_{p>m}\frac{F_*^{ii}\tau_{1pi}^2}{\lambda_1(\lambda_1 - \lambda_p)} + \sum_i \frac{F_*^{ii}\tau_{11i}^2}{\lambda_1^2}. 
    \end{aligned} \]
    Since, by concavity of $F^*$, we have  
    \[ \begin{aligned}
        \ddot{F_*}(\bnabla_1 \tau, \bnabla_1 \tau) &= \sum_{i,j}{F_*^{ii,jj}\tau_{ii1}\tau_{jj1}} - 2\sum_{p>q}\frac{F_*^{pp} - F_*^{qq}}{\lambda_q - \lambda_p}\tau_{1pq}^2 \leq  - 2\sum_{p>1}\frac{F_*^{pp} - F_*^{11}}{\lambda_1 - \lambda_p}\tau_{1p1}^2 ,
    \end{aligned} \]
    and 
    \[ - 2\sum_i\sum_{p>m}\frac{F_*^{ii}\tau_{1pi}^2}{(\lambda_1 - \lambda_p)} \leq -2\sum_{p>m}\frac{F_*^{11}\tau_{1p1}^2}{\lambda_1 - \lambda_p}, \]
    and \( \tau_{1p1} = 0\) for \(1 < p \leq m\) (by \eqref{eq:Brendle-lem}), we have 
    \[ \begin{aligned}
            0 &\leq -2\sum_{p>m}\frac{F_*^{pp}\tau_{11p}^2}{\lambda_1(\lambda_1 - \lambda_p)} - \frac{2F_*}{\lambda_1} + (Z+2) - \alpha Z  \\
            &\quad + \alpha \sum_i \frac{F_*^{ii}u_i^2}{u^2}  + \sum_i \frac{F_*^{ii}\tau_{11i}^2}{\lambda_1^2} \\
            &\leq -\sum_{p>m}\frac{F_*^{pp}\tau_{11p}^2}{\lambda_1(\lambda_1 - \lambda_p)} - \frac{2F_*}{\lambda_1} + (Z+2) - \alpha Z  \\
            &\quad + \alpha \sum_i \frac{F_*^{ii}u_i^2}{u^2} + \frac{F_*^{11}\tau_{111}^2}{\lambda_1^2}.
    \end{aligned} \]
    By the first order condition \( 0 = \phi_i\), we have \(\frac{u_i^2}{u^2} = \frac{\tau_{11i}^2}{\alpha^2 \lambda_1^2}\), then 
    \[ \begin{aligned}
            0 &\leq  (\alpha + \alpha^2) \frac{F_*^{11}u_1^2}{u^2} -\frac{2F_*}{\lambda_1} + (Z+2)-\alpha Z.
    \end{aligned} \]
    By Euler's identity \(\sum_{i=1}^{n}F_*^{ii}\lambda_i = F_*\) we have \(F_*^{11} \leq \frac{F_*}{\lambda_1}\), then
    \[ 0 \leq \alpha(\alpha+1)\frac{F_*|\bnabla u|^2}{\lambda_1 u^2} - ((\alpha-1)Z - 2). \]

    Now we show that \(Z \geq \delta >0\) when \(\lambda_1\) is sufficiently large. Note that \(F_*\) has a uniform upper bound by Proposition ~\ref{prop:normalized-estimates}. Assume that \(\lambda_1 \geq 2\sup F_*\), i.e. \(\lambda\) lies in the set
    \[ E = \{\lambda \in \Gamma_n \mid \lambda_1 \geq 2 \sup F_*, \quad\lambda_1\geq \lambda_2\geq \cdots \geq \lambda_n\}. \]
    Since both \(F_*\) and \(\lambda_1\) are homogeneous of degree \(1\), the functions \(g(\lambda) := F_*(\lambda)/\lambda_1\) and \(Z(\lambda)\) are homogeneous of degree \(0\). This allows us to restrict our analysis to the cross-section 
    \[ S = \{\lambda \in \Gamma_n \mid \lambda_1 = 1,\quad \lambda_1\geq \lambda_2\geq \cdots \geq \lambda_n\}.\] 
    Let \(e = (1,1,\cdots,1)\) be the umbilical point. Consider the sublevel set 
    \[ K = \{ \lambda \in S \mid g(\lambda) \leq \frac{1}{2} \}. \]
    Since \(g(e) = 1\), we have \(e\notin K\). Then by the continuity of \(g(\lambda)\), there exist a neighborhood \(B_{\eps}\ni e\) such that \(g\geq 3/4\) in \(B_{\eps}\). Now applying Lemma \ref{lem:algebraic} on \(f_*(\lambda) = \left(\frac{s_n}{s_{n-k}}\right)^{1/k}(\lambda)\) (with \(\Sigma = \{\sigma=(0,\sigma_2,\cdots,\sigma_n)\mid \sum_{i=2}^{n}\sigma_i^2 = 1\}\) there), we conclude that 
    \[ \delta := \inf_{\bar{\lambda} \in K} Z(\bar{\lambda}) > 0. \] 
    Clearly, for any \(\lambda\in E\), we have \(\bar{\lambda}: = \lambda / \lambda_1 \in K\). Then \(Z(\lambda) = Z(\bar{\lambda}) \geq \delta > 0\). 
    In conclusion, with the assumption \(\lambda_1 \geq 2\sup F_*\), we obtain \(Z\geq \delta > 0\). Then
    \[ 0 \leq \alpha(\alpha+1)\frac{F_*|\bnabla u|^2}{\lambda_1 u^2} - ((\alpha-1)\delta - 2). \]
    Take \(\alpha > 1 + \frac{2}{\delta}\), then 
    \[ \lambda_1 \leq \frac{C\alpha(1+\alpha)}{(\alpha-1)\delta - 2}. \]
    Therefore, 
    \[ \lambda_1 \leq \max \left\{ 2\sup F_*, \frac{C\alpha(1+\alpha)}{(\alpha-1)\delta - 2} \right\}. \]
    This finishes the proof.
\end{proof}

\begin{remark}
   Unlike \cite{WX20}, where the auxiliary functions is $x_{n+1}$, we deliberately utilize the normalized support function $\tilde{u}$ in Lemma \ref{lem:estimate-on-eigenvalues}. This is because $x_{n+1}$ is not $G$-invariant, while $\tilde{u}$ naturally preserves $G$-invariance. 
\end{remark}

Now we conclude the long-time existence of the unnormalized flow. 
\begin{theorem}\label{thm:long-time}
    If the initial spacelike hypersurface \(X_0\) is standard co-compact (Definition \ref{defn:std co-compact}) and strictly convex. Then the inverse \(\sigma_k\) curvature flow \eqref{eq:ICF} admits a long-time solution \(X: M\times [0,+\infty) \to \R^{n,1}\). Moreover, each \(X_t := X(\cdot, t)\) is standard co-compact.
\end{theorem}
\begin{proof}
    By the scaling relation $\tilde{\tau}_{ij} = e^t \tau_{ij}$, the uniform bound $\tilde{\lambda}_1 \leq C$ established in Lemma \ref{lem:estimate-on-eigenvalues} immediately implies an exponential decay $\lambda_1 \leq C e^{-t}$ for the unnormalized flow. Combining this upper bound with the algebraic fact that 
    \[ Ce^{-kt} \leq F^k_*(\tau) = \frac{s_n}{s_{n-k}} \leq c(n,k)\cdot \frac{\lambda_1\cdots\lambda_n}{\lambda_1 \cdots \lambda_{n-k}} \leq c(n,k)\cdot \lambda_1^{k-1}\lambda_n \]
    we obtain \(\lambda_n \geq \tilde{C} e^{-t}\). 
    Therefore the principal curvatures are bounded strictly away from $0$ and $+\infty$ on any finite time interval $[0, T]$. This excludes the development of finite-time singularities and completes the proof. 
\end{proof}

In the end of this section, we make some preparation for the proof of the convergence of the normalized flow. In general, for a Banach space \(X\), \(m\in \N\), we denote by \(C^{m}(I,X)\) the space consisting of all \(C^{m}\)-maps \(u\) from the interval \(I \subset \mathbb{R}\) to \(X\) with finite norm, where the norm of \(C^{m}(I,X)\) is defined by 
\begin{equation}
    \|u\|_{C^{m}(I,X)} := \sup_{t\in I ; 0\leq l \leq m}\|\partial_t^{l}u(t)\|_{X}.
\end{equation} 
In particular, for \(X = C^{k}(\H^n / G)\), \(I = [0,T]\) and \(u: (\H^n / G)\times [0,T] \to \R\), the norm is given by
\begin{equation}
    \|u\|_{C^{m}\left([0,T],C^{k}(\H^n / G)\right)} := \sup_{t\in[0,T] ; 0\leq l \leq m}\|\partial_t^{l}u(\cdot,t)\|_{C^{k}(\H^n / G)}.
\end{equation}
Similarly, for \(X = C^{k,\alpha}(\H^n / G)\) (\(0 < \alpha < 1\)) and \(u: (\H^n / G)\times [0,T] \to \R\),
\begin{equation}
    \|u\|_{C^{m}\left([0,T],C^{k,\alpha}(\H^n / G)\right)} := \sup_{t\in[0,T] ; 0\leq l \leq m}\|\partial_t^{l}u(\cdot,t)\|_{C^{k,\alpha}(\H^n / G)}.
\end{equation}
\par Having established the global existence of the flow, we now proceed to obtain higher-order estimates. Combining all the estimates above (Proposition ~\ref{prop:normalized-estimates} and Lemma ~\ref{lem:estimate-on-eigenvalues}), by Evans-Krylov theory, for any \(T > 0\), the normalized support function \(u(y,t)\) satisfies
\[ \sup_{t\in(0,T)}\|u(\cdot,t)\|_{C^{2,\alpha}(\H^n / G)} \leq C, \] 
where \(C>0\) is a uniform constant independent of \(T\). Then by \eqref{eq:normalized-flow} we have 
\begin{equation}
    \|u\|_{C^{1}\left([0,T],C^{2,\alpha}(\H^n / G)\right)} \leq C.
\end{equation}
It follows a priori estimates for higher-order derivatives immediately. 
\begin{corollary}\label{cor:higer-order estimates}
    Let \(u(y,t)\) satisfies the normalized flow \eqref{eq:normalized-flow} with \(u(\cdot,0) = u_0(\cdot)\), then for any integer \(k\geq 2\), any \(T > 0\), we have
    \begin{equation}
        \|u\|_{C^{1}([0,T],C^{k}(\H^n / G))} \leq C,
    \end{equation}
    where \(C > 0\) is a constant depending on \(n,k,\|u_0\|_{C^{2}(\H^n / G)}\), and independent of \(T\).
\end{corollary}

\section{Convergence}\label{sec:convergence}
In this section we will show the convergence of the normalized flow. By Lemma \ref{lem:normalized-equation-for-u}, the normalized evolution equation is 
\[ \mathscr{L}u = -Z u, \]
where \(Z = \dot{F_*}(\bg)-1 \geq 0 \), and \(\mathscr{L} = \partial_t - \dot{F}_*^{kl}\bnabla_k \bnabla_l \). Since we already have the uniform estimate on the coefficients of \(\mathscr{L}\), independent of \(T\), the following strong maximum principle (\cite{Lieberman1996} Theorem 2.7) holds: 
\begin{lemma}[Strong Maximum Principle]
  If \(w(x,t)\) satisfies 
    \[ \mathscr{L}w \leq c^i(x,t) \bnabla_i w, \]
    Then the supremum of \(w(x,t)\) is attained on \(M\times \{t=0\}\). Moreover, if \(w(x,t)\) attains its supremum in the interior of \(M \times (0,T)\) for some \(T > 0\), then \(w\) must be constant in \(M \times (0,T)\). 
\end{lemma}

We first show that the right hand side \( -F_* + u\) of the equation \eqref{eq:normalized-flow} will converge to \(0\) as \(t\to +\infty\), by using a similar idea to \cite{Huang2021}. The key observation here is that \(u\cdot F_*^{-1}\) (or \(\log u - \log F_*\) equivalently) satisfies the above maximum principle. 
\begin{lemma}\label{lem:logu-logF}
    Let \(w = \log u - \log F_*\), then \(w \to \text{const}\), locally uniformly as \(t\to +\infty\). 
\end{lemma}
\begin{proof}
    Choose a local orthonormal frame \(\{e_1,\cdots,e_n\}\), 
    \[ w_i = \frac{u_i}{u} - \frac{(F_*)_i}{F_*}, \qquad w_t = \frac{u_t}{u} - \frac{(F_*)_t}{F_*}, \]
    \[ w_{ii} = \frac{u_{ii}}{u} - \frac{(u_i)^2}{u^2} - \frac{(F_*)_{ii}}{F_*} + \frac{(F_*)_i^2}{F_*^2}, \]
    \begin{equation}\label{eq:eq-for-w}
        \begin{aligned}
        \implies  \mathscr{L}w &= \frac{\mathscr{L}u}{u} -\frac{\mathscr{L}{F_*}}{F_*} + \sum_i \frac{F_*^{ii}u_i^2}{u^2} - \sum_i \frac{F_*^{ii}(F_*)_i^2}{F_*^2}\\
        &= -Z + Z + \sum_i F_*^{ii} \left( \frac{u_i}{u} +\frac{(F_*)_i}{F_*}\right) \left( \frac{u_i}{u} - \frac{(F_*)_i}{F_*} \right) \\
        &=: \sum_i c^i w_i.
        \end{aligned} 
    \end{equation}
    Thus by the strong maximum principle, \(w\) attains its supremum and infimum for \((\H^n / G) \times [t_0,+\infty)\) only on the boundary \((\H^n / G) \times \{t=t_0\}\). Then we can further conclude that 
    \[ osc(w)(t) := \max_{z\in \H^n}w(z,t) - \min_{z\in \H^n}w(z,t) \]
    is strictly decreasing. This monotonicity leads to a convergence as \(t\to +\infty\):
    \[ \lim_{t\to +\infty} osc(w)(t) = \epsilon \geq 0 \]
    If \(\epsilon = 0\), then the lemma has been proved. Assume not, that is, \(\epsilon > 0\). Fix \(T > 0\). For any sequence \(\{t_n\}\to +\infty (n\to \infty)\), define a sequence of functions,  
    \[ w_n(z,t) := w(z,t+t_n), \quad z\in \H^n, t\in (-t_n,+\infty). \]
    By Corollary~\ref{cor:higer-order estimates}, we have 
    \(\norm{w_n}_{C^1((-t_n,T), C^{2,\alpha}(\H^n / G))}\leq C\), where \(C>0\) is a uniform constant independent of \(T\). By the Arzelà-Ascoli theorem, there exists a subsequence of \(\{t_n\}\), again denoted by \(\{t_n\}\), and a limit function \(w_{\infty}^{T}(z,t)\) such that 
    \[ w_n(z,t) \to w_{\infty}^{T}(z,t), \quad \text{in }C^0\left([-T,T],C^{2,\beta}(\H^n / G)\right) \]
    for some \(0 < \beta < \alpha\). By a standard Cantor's diagonal argument, there is a subsequence of \(\{t_n\}\), again denoted by \(\{t_n\}\), and a limit function \(w_{\infty}(z,t)\) such that 
    \[ w_n(z,t) \to w_{\infty}(z,t), \quad \text{in }C_{loc}^{0}\left((-\infty,+\infty),C^{2,\beta}(\H^n / G)\right) \]
    for some \(0 < \beta < \alpha\). Here, \(w_n\) converges to \(w_{\infty}\) in \(C_{loc}^{0}\left((-\infty,+\infty),C^{2,\beta}(\H^n / G)\right)\) means that for any \(M > 0\), we have 
    \[ w_n(z,t) \to w_{\infty}(z,t), \quad \text{in }C^{0}\left([-M,M],C^{2,\beta}(\H^n / G)\right). \] 
    Then \(w_\infty(z,t)\) also satisfies equation \eqref{eq:eq-for-w}. Notice that \( osc(w_n)(t) = osc(w)(t+t_n)\), then 
    \[ osc(w_\infty)(t) = \lim_{n\to \infty}osc(w)(t+t_n) = \lim_{t\to +\infty}osc(w)(t) = \epsilon > 0, \quad \forall t\in (-\infty,+\infty). \]
    However, by the strong maximum principle, \(osc(w_\infty)(t)\) should be strictly decreasing until it reaches zero, which is a contradiction. 
\end{proof}

\begin{corollary}
    Let \(F_*\) and \(u\) be the speed function and support function of the rescaled solution \(\tilde{X}(p,t) = e^t X(p,t)\), then 
    \[ -F_* + u \to 0,\quad \text{locally uniformly as }t\to +\infty. \]
\end{corollary}
\begin{proof}
    We first conclude that \(u F_*^{-1} \to 1\). If not, by Lemma \ref{lem:logu-logF}, \(w = \log u - \log F_* \to const\), suppose that \(u F_*^{-1} \to c\neq 1\). Then \(u_t = -F_* + u = F_*(uF_*^{-1} - 1) \to F_*(c-1) \) as \(t \to +\infty\). Note that \(0 < C_1 \leq F_* \leq C_2\), if \(c > 1\), then \(\exists T_1 > 0\) such that \(u_t \geq \frac{C_1 (c-1)}{2} > 0 \) when \(t > T_1\). This implies \(u \to +\infty\) as \(t\to +\infty\), which contradicts with \(u \leq C\). Similarly, \(c < 1\) will leads to \(u \to -\infty\), which is also a contradiction. Therefore, \(u F_*^{-1} \to 1\). By the uniform bounds for \(F_*\), we conclude that \(-F_* + u = F_*(uF_*^{-1} - 1) \to 0\).
\end{proof}

Now we can prove the convergence of the flow. \par
\textbf{Proof of Theorem }{\ref{thm:convergence}}: 
We first show that the volume of the rescaled hypersurface converges. In fact, the normalized ICF is volume decreasing: the normalized volume element \(\dd \tilde{\mu}_t = \sqrt{\det \tilde{g}}\dd x = e^{nt} \sqrt{\det g}\dd x\) satisfies 
\[ \begin{aligned}
    \pd{\det \tilde{g}}{t} &= \pd{}{t}\left(e^{2nt} \det g \right) \\
    &= e^{2nt}\left( \pd{\det g}{t} + 2n\det g \right) \\
    &= e^{2nt}\left( \pd{}{t}g_{ij} \cdot g^{ij} \cdot \det g + 2n\det g \right) \\
    &= 2e^{2nt}\det g \left( -\frac{H}{s_k^{1/k}} + n \right) \\
    &\leq 0,
\end{aligned}  \]
where we used the Newton-Maclaurin inequality \( s_1(\kappa)\geq s_k^{1/k}(\kappa) \). Especially, \(\pd{\det \tilde{g}}{t} \equiv 0\) when \(S = s_1^{-1}\), i.e. \(F(\kappa) = s_1(\kappa)\). Therefore, the volume of \(M / G\) converges to a constant \(V^\infty \geq 0\)(later we will see \(V^\infty \neq 0\)), uniquely determined by the initial co-compact hypersurface and the speed function \(F(\kappa)\).  
\par By Corollary~\ref{cor:higer-order estimates}, the normalized support function \(u(z,t)\) satisfies that for any \(k \geq 1\), there exist a positive constant \(C_k>0\) such that 
\[ \norm{u}_{C^1\left( (0,+\infty),C^k(\H^n / G) \right)} \leq C_k.\]
By Arzela-Ascoli theorem and a diagonal subsequence argument, for any \(\{t_n\}\to +\infty\), there exist a subsequence \(\{u(\cdot,t_{n_k})\}_{k=1}^{\infty}\) converges in \(C^{\infty}\)-topology to a limit function \(u^{\infty}(\cdot)\). Since \(u_t = -F_* + u \to 0\) as \(t\to +\infty\), any convergent subsequence of \(u(\cdot,t)\) will converge to a self-shrinker whose support function denoted by \(u^{\infty}(\cdot)\). By the rigidity (Theorem ~\ref{lem:rigidity}), we conclude that \(u^{\infty}(\cdot) \equiv r^\infty\) is a constant. Moreover, \(r^{\infty}\) and the corresponding limit spacelike hypersurface \(M^{\infty}\) should satisfy the volume relation: 
\begin{equation}\label{eq:volume}
0 < \operatorname{vol} (M^{\infty}/G) = V^{\infty} \leq \operatorname{vol} (M_0/G).
\end{equation}
Especially, the equality holds when \(F(\kappa) = s_1(\kappa)\). Hence \(u(\cdot,t)\) converges in \(C^{\infty}\)-topology to the constant \(r^{\infty}\) which satisfies \eqref{eq:volume}. \qed

\section{Application on Geometric Inequalities}
In this section, we prove Theorem \ref{thm:quermassintegral_ineq}. We define the isoperimetric ratio 
\begin{equation}
    \mathcal{I}_k(\overline{M}) = \frac{\abs{\overline{M}}^{\frac{1}{n}}}{\left( \int_{\overline{M}} \sigma_k(\kappa) \dd \mu_M \right)^{\frac{1}{n-k}}}
\end{equation}
This ratio is scaling-invariant: if \(\widetilde{M}\) is defined by a scaling of the position vector of \(M\), then \(\mathcal{I}_k(\widetilde{M}) = \mathcal{I}_k(M)\). 

Our goal is  to prove that for any standard co-compact convex spacelike hypersurface \(M\), we have
\begin{equation}
    \mathcal{I}_k(M / G) \geq \mathcal{I}_k(\H^n / G)
\end{equation}
where \(G\) is the Lorentz subgroup such that \(M / G\) is compact. 

\subsection{Some Calculations}
Let \(S(h,g) = -F^{-1}(h,g)\). By proposition \ref{prop:evolution_geo_quantaties}, we have 
\begin{equation}
    \pd{}{t}\dd \mu_g = -S\sigma_1(\kappa)\dd \mu_g,
\end{equation}
\begin{equation}
    \pd{}{t}h_j^i = -\nabla^i\nabla_j S + S(h^2)_j^i,
\end{equation}
and for \(1\leq m\leq n\),
\begin{equation}
    \begin{aligned}
        \pd{\sigma_m}{t} &= \sum_{i,j=1}^{n} \pd{\sigma_m}{h_j^i}\cdot\pd{}{t}h_j^i \\
        &= \sum_{i,j=1}^{n} \pd{\sigma_m}{h_j^i}\left( -\nabla^i\nabla_j S + S(h^2)_j^i \right) \\
        &= -\sum_{i,j=1}^{n} \nabla^i\left(\pd{\sigma_m}{h_j^i} \nabla_j S \right) + S\sum_{i,j=1}^{n} \pd{\sigma_m}{h_j^i}(h^2)_j^i,
    \end{aligned}
\end{equation}
where we have used the divergence-free property of \(\pd{\sigma_m}{h_j^i}\)(see \cite{R.Reilly1974} for reference).

\begin{lemma}\label{eq:evolution_integral_sigma_l}
    Under the flow 
    \begin{equation}\label{eq:inverse_sigma_k_flow}
        \pd{}{t}X(p,t) = -S(p,t)\nu(p,t)
    \end{equation}
    of a family of spacelike hypersurfaces \(X : M^n\times (0,+\infty) \to \R^{n,1}\), we have the following evolution equation: 
    \begin{equation}
        \partial_t \int_{\overline{M}} \sigma_l \dd \mu_g = -(l+1)\int_{\overline{M}} S\cdot \sigma_{l+1} \dd \mu_g,\quad l=1,2,\cdots,n-1
    \end{equation}
    where \(g\) is the induced metric on \(M\) from the ambient space. 
\end{lemma}
\begin{proof}
    \begin{equation}
        \begin{aligned}
            \partial_t \int_{\overline{M}} \sigma_l \dd \mu_g &= \int_{\overline{M}} (\partial_t \sigma_l) d\mu_g + \int_{\overline{M}} \sigma_l \partial_t(\dd \mu_g) \\
            &= \int_{\overline{M}} \left( -\sum_{i,j=1}^{n}\nabla^i \left(\pd{\sigma_l}{h_j^i}\nabla_j S\right) + S\sum_{i,j=1}^{n} \pd{\sigma_m}{h_j^i}(h^2)_j^i \right) \dd \mu_g \\
            &\quad + \int_{\overline{M}} \sigma_l(-S\sigma_1 \dd \mu_g) \\
            &= \int_{\overline{M}} S\left( \sum_{i,j=1}^{n} \pd{\sigma_l}{h_j^i}(h^2)_j^i - \sigma_1(\kappa)\sigma_l(\kappa) \right)\dd \mu_g \\
            &= -(l+1)\int_{\overline{M}} S\cdot \sigma_{l+1}\dd \mu_g.
        \end{aligned}
    \end{equation}
\end{proof}

\par {\bf Proof of Theorem \ref{thm:quermassintegral_ineq}:} Let \(X(p,t)\) be the solution of the flow \eqref{eq:inverse_sigma_k_flow} with \(S = \frac{1}{s_k^{1/k}(\kappa)}\) and initial surface \(M_0 = M\). Consider \(\widetilde{X} = e^{t}X\). Then the corresponding hypersurface of \(\widetilde{X}\) is still co-compact. Denote \(\widetilde{M} = \widetilde{X} / G\). By the scaling-invariance, 
\[ \mathcal{I}_k(\widetilde{M}) = \mathcal{I}_k(M), \quad \forall t \geq 0. \]
By Lemma \ref{eq:evolution_integral_sigma_l}, we have
\begin{equation}
    \begin{aligned}
        \partial_t \int_{\widetilde{M}} \sigma_l(\widetilde{\kappa}) \dd \widetilde{\mu_g} &= \partial_t \left( e^{(n-l)t}\int_{M} \sigma_l(\kappa) \dd \mu_g \right) \\
        &= e^{(n-l)t}\cdot \left((n-l)\int_M \sigma_l \dd \mu_g - (l+1)\int_{M}S\cdot \sigma_{l+1} \dd \mu_g \right).
    \end{aligned}
\end{equation}
Setting \(l=0,k\) and \(S = \frac{1}{s_k^{1/k}}\), we obtain
\begin{equation}
\begin{aligned}
    \partial_t \int_{\widetilde{M}} \sigma_k(\widetilde{\kappa}) \dd \widetilde{\mu_g} &= e^{(n-k)t}\left( (n-k)\int_{M}\sigma_k - (k+1) \int_{M}\frac{\sigma_{k+1}}{s_k^{1/k}} \right) \\
    &= e^{(n-k)t}\frac{n!}{(n-k-1)!k!}\left( \int_{M}s_k - \int_{M}\frac{s_{k+1}}{s_k^{1/k}} \right) \geq 0,
\end{aligned}
\end{equation}
and 
\begin{equation}\label{eq:integral_sigma_0_ineq}
    \begin{aligned}
        \partial_t \int_{\widetilde{M}} \sigma_0(\widetilde{\kappa}) \dd \widetilde{\mu_g} &= e^{nt}\left( n\int_M 1\dd\mu_g - \int_M \frac{\sigma_1}{s_k^{1/k}}\dd \mu_g \right) \\
        &= e^{nt}n\left( \int_M 1\dd\mu_g  - \int_M \frac{s_1}{s_k^{1/k}}\dd \mu_g \right) \leq 0,
    \end{aligned}
\end{equation}
where we have used the following Newton-Maclaurin inequality:
\[ s_1 \geq s_k^{1/k} \geq s_{k+1}^{1/(k+1)},\quad 1\leq k \leq n-1 . \]
Then we conclude that \(\mathcal{I}_1(\widetilde{M})(t)\) is decreasing in \(t\), and the desired inequality \eqref{eq:geo-ineq} follows from Theorem \ref{thm:convergence}. 
\par If the equality holds in \eqref{eq:geo-ineq}, we must have 
\[ \partial_t \int_{\widetilde{M}} \sigma_0(\widetilde{\kappa}) \dd \widetilde{\mu_g} \equiv 0. \]
Therefore, the equality of the Newton-Maclaurin inequality must be held at every point of \(M\) in \eqref{eq:integral_sigma_0_ineq}. This implies \(M\) is umbilical for each \(t \geq 0\). In particular, \(M_0\) is a part of \(\H^n\). \qed

\bigskip

Following Guan and Yuan's argument \cite{guan2025-intrinsicality}, we present the following lemma: 
\begin{lemma}\label{lem:intrincity}
    Let \(n \geq 3\), \((M^n,g)\) be a strictly convex spacelike hypersurface in Minkowski space \(\R^{n,1}\), then \(\sigma_k(\kappa[M])\) are intrinsic geometric quantities for all \(1\leq k \leq n\). 
\end{lemma}
\begin{proof}
    Consider an arbitrary point \(x_0\in M\) and choose a orthonormal frame \(\{e_i\}_{i=1}^n\) around \(x_0\), such that \((g_{ij}) = (\delta_{ij})\), and \( (h_{ij})= \diag{\kappa_1,\cdots,\kappa_n}\) at \(x_0\). Set \({R}(X,Y)Z = -\nabla_X\nabla_Y Z + \nabla_Y\nabla_X Z + \nabla_{[X,Y]}Z\), and \({R}_{ijkl} = g(R(e_i,e_j)e_k,e_l)\), \( {R}_{ikj}^{\quad p} = {R}_{ikjm}g^{mp} \). By Gauss equation in Minkowski space \eqref{eq:Gauss-eq}, 
    \[ R_{ijkl} = -(h_{ik}h_{jl} - h_{il}h_{jk}), \]
    which means \(R_{ijij} = - \kappa_i \kappa_j\) for \(i\neq j\). Then for \(k=2m\) (\(m \geq 1\)), \(\sigma_k(\kappa)\) is a combination of \(\kappa_\alpha \kappa_\beta\), therefore a combination of \(R_{\alpha\beta\alpha\beta}\), which is intrinsic. 
    \par For \(k = 2m+1\) (\(m\geq 1\)), 
    \[ \sigma_{2m+1}^2(\kappa) = \frac{1}{2m+1}\sum_{i=1}^{n}\sigma_{2m}(\kappa|i)\kappa_i \sigma_{2m+1}(\kappa). \]
    For each \(i\), 
    \[ \sigma_{2m}(\kappa|i)\kappa_i \sigma_{2m+1}(\kappa) = \sigma_{2m}(\kappa|i)[\kappa_i\sigma_{2m+1}(\kappa|i) + \kappa_i^2\sigma_{2m}(\kappa|i)]. \]
    Note that \(\sigma_{2m}(\kappa|i)\) is a combination of \(\kappa_\alpha \kappa_{\beta}\) for \(\alpha\neq \beta \neq i\), and also \(\kappa_i\sigma_{2m+1}(\kappa|i)\) is a combination of \(\kappa_\alpha\kappa_\beta\) and \(\kappa_\gamma\kappa_i\) for \(\alpha\neq \beta\neq i\) and \(\gamma\neq i\). Then both \(\sigma_{2m}(\kappa|i)\) and \(\kappa_i\sigma_{2m+1}(\kappa|i)\) can be expressed as a function of sectional curvatures \(R_{\alpha\beta\alpha\beta}\) and \(R_{\gamma i \gamma i}\). Moreover, for \(m\geq 1\), \(\kappa_i^2 \sigma_{2m}(\kappa|i)\) can be expressed in terms of \(\kappa_\alpha \kappa_\beta\), \(\kappa_\alpha\kappa_i\), and \(\kappa_\beta\kappa_i\) for \(\alpha\neq\beta\neq i\), which are all sectional curvatures of the metric \(g\). In conclusion, \(\sigma_{2m+1}(\kappa)\) (\(m\geq 1\)) are polynomials of the sectional curvatures, which are intrinsic. 
    \par For \(k=1\), 
    \[ \sigma_1^2(\kappa) = \sum_{i=1}^{n}\kappa_i^2 + 2\sigma_2(\kappa). \]
    Since \(\sigma_2(\kappa)\) is a polynomial of sectional curvatures, we focus on the first term
    \[ |\kappa|^2 = \sum_{i=1}^{n} \kappa_i^2 = \frac{1}{\sigma_3^2(\kappa)}\sum_{i=1}^{n}[\kappa_i \sigma_3(\kappa)]^2. \]
    For each \(i\), we have
    \[ \kappa_i \sigma_3(\kappa) = \kappa_i\sigma_3(\kappa|i) + \kappa_i^2\sigma_2(\kappa|i). \]
    Note that the right hand side is a combination of \(\kappa_\alpha\kappa_\beta\) and \(\kappa_\gamma\kappa_i\) for \(\alpha\neq \beta\neq i\) and \(\gamma\neq i\). Then \(\kappa_i \sigma_3(\kappa)\) is a combination of sectional curvatures \(R_{\alpha\beta\alpha\beta}\) and \(R_{\gamma i \gamma i}\). Combined with the intrinsicality of \(\sigma_3(\kappa)\) as we have already proved, one sees the intrinsicality of \(|\kappa|^2\). 
\end{proof}


\end{document}